\documentclass[12pt,a4paper]{article}
\usepackage{amsmath,amssymb,amsthm}
\usepackage{enumitem}
\usepackage{hyperref} 

\newtheorem{theorem}{Theorem}
\newtheorem{lemma}[theorem]{Lemma}
\newtheorem{corollary}[theorem]{Corollary}
\newtheorem{proposition}[theorem]{Proposition}
\theoremstyle{remark}
\newtheorem{remark}[theorem]{Remark}

\newcommand{\eps}{\varepsilon}
\newcommand{\calL}{\mathcal{L}}

\newcommand{\K}{\mathbb{K}}
\newcommand{\N}{\mathbb{N}}
\newcommand{\R}{\mathbb{R}}

\newcommand{\norm}[1]{\lVert #1 \rVert}

\newcommand{\ipr}[1]{\left\langle#1\right\rangle}
\DeclareMathOperator{\codim}{codim}

\DeclareMathOperator{\rk}{rank}
\DeclareMathOperator{\rank}{rank}

\DeclareMathOperator{\diag}{diag}
\newcommand{\id}{I}

\newcommand{\set}[1]{\left\{#1\right\}}
\newcommand{\bset}[1]{\Bigl\{#1\Bigr\}}

\newcommand{\co}{\colon}
\newcommand{\ds}{\displaystyle}

\title{On bounds between all s-numbers \\and widths of convex sets}
\author{Mario Ullrich}
\date{\today}

\begin{document}
\maketitle

\begin{abstract}
We prove $a_n(S) \le e\,(n+1)\, s_n(S)$ 
for every s-number sequence~$(s_n)$, every bounded linear operator $S$ between normed spaces, and every $n \in \N_0$, 
where $a_n$ are the approximation numbers, which are the largest s-numbers. 
This is sharp up to the constant and settles conjectures of 
Mityagin, Henkin, Carl and Pietsch dating back to 1963. 
We also extend it to widths of convex sets and discuss optimality there. The proof is elementary.
\end{abstract}

\section{Introduction}

The concept of 
s-numbers, as axiomatized by Pietsch~\cite{Pietsch87,Pietsch07,Pietsch-s}, generalizes singular values of operators (or matrices) mapping between Hilbert spaces, to operators between Banach spaces. 
Some examples of s-numbers are motivated by the various equivalent definitions of singular values or by geometric considerations, and some appear naturally in approximation theory.  

For example, for a bounded linear operator $S$ between Banach spaces, the approximation numbers $a_n(S)$ measure the error of best rank-$n$ approximation; the Gelfand and Kolmogorov numbers $c_n(S)$ and $d_n(S)$ describe optimal recovery errors from $n$ linear measurements and best approximation from $n$-dimensional subspaces; the Bernstein numbers $b_n(S)$ measure the largest $(n+1)$-dimensional ball in the image of the unit ball; the manifold numbers $\delta_n(S)$, which are genuinely nonlinear, measure the best approximation by continuous mappings through an $n$-dimensional manifold; and the Hilbert numbers $h_n(S)$ are the singular values of the ``largest'' Hilbert space operator that factors through $S$. The s-number axioms (given below) are designed so that 
they all 
agree on Hilbert spaces with the singular values.

Some of these s-numbers
are fundamental objects in 
approximation theory~(\cite{CDPW,DHM89,Ism74,Pinkus,Tikh60}) and information-based complexity (\cite{KNU,KU26,Mathe90,Novak-widths,TWW88}). 
They are also crucial in the study of eigenvalues of operators, operator ideals, and other properties of Banach spaces, see e.g.~\cite{Koe86,LE11,PT89,Pietsch87,Pietsch07,Pietsch-ideals}.

One advantage of an axiomatic theory of s-numbers is the characterization of the smallest and largest among them, 
possibly with additional conditions. 
In particular, 
the approximation numbers are the largest 
and the Hilbert numbers are the smallest s-numbers, 
so that 
\[
h_n(S) \le s_n(S) \le a_n(S)
\] 
for every s-number sequence~$(s_n)$, 
see~\cite[2.3.4 \& 2.6.3]{Pietsch87}. 
Clearly, reverse bounds are desirable, and 
it 
has been proven by Mityagin and Henkin~\cite{MH} in 1963 that 
\begin{equation} \label{eq:MH}
d_n(S) \le (n+1)^2\, b_n(S),
\end{equation}
see also~\cite[Theorem~8.6]{Pietsch-s}. 
They also conjectured that the factor $n+1$ suffices. 
Bauhardt~\cite{Bauhardt} gave the corresponding bound and conjecture with $h_n$ in place of $b_n$. The factor $(n+1)^2$ has apparently not been improved since, and the conjecture is repeated in the books of 
Pietsch~\cite[11.12.4]{Pietsch-ideals} 
and Pinkus~\cite[II.5]{Pinkus}, see also Novak~\cite[eq.~(2.9)]{Novak-widths}.
On top of that, Carl and Pietsch~\cite{CP78} conjectured that one may even 
have $a_{bn}(S)\le c\, (n+1)\, h_n(S)$, for some $b,c\ge1$, 
see also 
Pietsch's ``Long-standing open problems of Banach space theory: My personal top ten''~\cite[Problem~5]{Pie09}, 
or~\cite[6.2.3.14]{Pietsch07}.

In a weaker form, 
the conjecture of Mityagin and Henkin has been proven by Pietsch~\cite[11.12.3]{Pietsch-ideals}, who showed that 
\begin{equation} \label{eq:bound-geom}
\max\bset{c_n(S),\, d_n(S)} 
\;\le\; (n+1)\,\left(\prod_{k=0}^n h_k(S) \right)^{1/(n+1)}, 
\end{equation}
see also the streamlined presentation of~\cite{U24} for the improved constant. (Note that there is an index shift compared to some definitions in the literature.) \\
A comparison of all s-numbers was then achieved 
by a known comparison between $a_n$ and $c_n$/$d_n$ based on projections of small norm of~\cite{GarlingGordon,KS71}, see~\cite[Theorem~8.4]{Pietsch-s}, leading to an additional factor~$\sqrt{n+1}$. 

In Theorem~\ref{thm:main}, we prove the 
stronger conjecture of Carl and Pietsch up to the constant $e$, 
i.e., 
\[
	a_n(S)
	\;\le\; e\,(n+1)\, h_n(S) 
\]
for every bounded linear operator $S$ and every $n \in \N_0$. 
%
%

\medskip

The proof is based on a new determinant quantity~$\Delta_k(S)$, see Section~\ref{sec:delta}, 
which we use instead of the \emph{Grothendieck determinant} 
that is usually used in this context. 
We show that the decay of $\Delta_k(S)$
can be controlled by both $h_n$ (Lemma~\ref{lem:ratio-u}) and in terms of 
$a_n$
(Lemma~\ref{lem:ratio-l}), 
which leads to the main result.

\medskip

For $S\colon X\to Y$ with $X$ or $Y$ 
a Hilbert space, 
Mityagin and Henkin~\cite{MH} also conjectured that the factor $(n+1)^2$ in~\eqref{eq:MH} can be replaced by~$\sqrt{n+1}$. 
Up to a factor $\sqrt{e}$, 
this conjecture has been settled by Pukhov~\cite{Pu79}.  
See also~\cite[Theorem 3.6]{KNU} for a 
bound between $c_n$ and $b_n$ without the $\sqrt{e}$. 
Here, we extend it to a bound between all s-numbers, showing that 
\[
a_n(S) \le \sqrt{e\,(n+1)}\; h_n(S)
\]
whenever $X$ or $Y$ is a Hilbert space, 
see Theorem~\ref{thm:main}.

\medskip

We add that many s-numbers can be understood as certain \emph{widths of sets}. 
In particular, by considering the identity $S=\id\colon X\to Y$ with $\id x=x$ and $X\subset Y$, one obtains by $s_n(\id)$ various widths of the unit ball~$B_X$, measured in the norm of~$Y$.  
We discuss an extension to width of (nonsymmetric) convex sets in Section~\ref{sec:sets}.
The same ideas for the bound between them apply 
in this case. 
While Theorem~\ref{thm:main} carries over verbatim for convex and symmetric $F$, 
an additional factor $\sqrt{n+1}$ appears 
for nonsymmetric $F$, see Theorem~\ref{thm:sets}. 

\section{Setting and main result}\label{sec:setting}

Throughout, $\K \in \{\mathbb{R},\mathbb{C}\}$, and $X, Y, W, Z$ are normed spaces over $\K$.
We write $B_X$ for the closed unit ball of $X$,
$I_X$ for the identity on $X$, and $X'$ for the continuous dual of $X$. 
For $x\in X$ and a functional $b\in X'$, we write $\ipr{x,b}=b(x)$. 
By $\mathcal{L}(X,Y)$ we denote the bounded linear operators from~$X$ to~$Y$, with $\mathcal{L}$ being the class of all such operators. 
An operator $A\in\calL$ with $\|A\|\le1$ is called \emph{contraction}. 
For a Hilbert space~$H$
and denote by $\ipr{\cdot,\cdot}_H$ its inner product, and for $A\in\calL(H,K)$ between Hilbert spaces we write $A^*\in\calL(K,H)$ for its adjoint, defined by
$\ipr{Ax,y}_K=\ipr{x,A^*y}_{H}$ for $x\in H$ and $y\in K$. 
Moreover, $\ell_p^k$ is $\K^k$ with usual $p$-norm and standard basis $(e_k)$.

Following \cite{Pietsch07}, a map $S \mapsto (s_n(S))_{n \in \N_0}$ assigning to every $S \in \mathcal{L}$ a nonnegative scalar sequence is an \emph{s-number sequence} if 
for all $n \in \N_0$:
\begin{enumerate}[label=\textup{(S\arabic*)}, leftmargin=3em]
\item $\norm{S} = s_0(S) \ge s_1(S) \ge \ldots \ge 0$ \quad for all $S \in \mathcal{L}$,
\item $s_n(S + T) \le s_n(S) + \norm{T}$  \quad for all $S, T \in \mathcal{L}(X,Y)$,
\item $s_n(BSA) \le \norm{B}\, s_n(S)\, \norm{A}$  \quad for
$W \xrightarrow{A} X \xrightarrow{S} Y \xrightarrow{B} Z$,
\item $s_n(S) = 0$  \quad whenever $\rk(S) \le n$,
\item $s_n(I_{\ell_2^{n+1}}) = 1$.
\end{enumerate}

Indeed, for a compact operator $S$ between Hilbert spaces the axioms determine $s_n(S)$ uniquely as the $n$-th \emph{singular value} $\sigma_n(S) := \sqrt{\lambda_n(S^*S)}$, 
where the eigenvalues $\lambda_k(T)$ of $T\in\calL(X,X)$ are characterized by $Tv_k=\lambda_k(T)\cdot v_k$ for some $v_k\in X\setminus\{0\}$, and ordered decreasingly,  
see~\cite[2.11.9]{Pietsch87}. 

We index singular values, like s-numbers, starting from zero, 
so that the smallest singular value of an $(n+1) \times (n+1)$ matrix $T$ is $\sigma_n(T)$. 
Note that 
singular numbers and s-numbers are often indexed from one, but we follow the convention of \cite{Pinkus} and \cite{KU26}. 

\medskip

Important examples are the approximation, Bernstein, Gelfand, Kolmogorov, 
manifold and Hilbert numbers: 
\begin{align*} 
  a_n(S) &\;:=\; \inf\Bigl\{\, \norm{S - L} \co\; L \in \mathcal{L}(X,Y),\ \rk L \le n \,\Bigr\},\notag\\[2pt]
  b_n(S) &\;:=\; \sup\Bigl\{\, \inf_{\substack{x \in M \\ x \ne 0}} \tfrac{\norm{Sx}}{\norm{x}}
      \co\; M \subseteq X \text{ subspace},\ \dim M = n+1 \,\Bigr\},\notag\\[2pt]
  c_n(S) &\;:=\; \inf\Bigl\{\, \norm{S|_N} \co\; N \subseteq X \text{ closed subspace},\ \codim N \le n \,\Bigr\},\notag\\[2pt]
  d_n(S) &\;:=\; \inf\Bigl\{\, 
  \sup_{x \in B_X}\, \inf_{y \in V} \norm{Sx - y}
  \co\; V \subseteq Y 
  \text{ subspace},\ 
  \dim V \le n \,\Bigr\},\notag\\[2pt]
  \delta_n(S) &\;:=\; \inf\ \Bigl\{\, \sup_{x \in B_X} \norm{ Sx - \Phi(\psi(x)) }
      \co\; \psi \in C(B_X, \K^n),\ \Phi \in C(\K^n, Y) \,\Bigr\},\notag\\[2pt]
  h_n(S) &\;:=\; \sup\Bigl\{\, a_n(BSA) \co\; A \in \mathcal{L}(\ell_2, X),\ B \in \mathcal{L}(Y, \ell_2),\ \norm{A}, \norm{B} \le 1 \,\Bigr\}, 
\end{align*}
where 
$C(M_1,M_2)$ denotes the set of continuous maps from $M_1$ to $M_2$. 
Detailed proofs of the \textit{s-number axioms} and further properties can be found in~\cite{Pietsch87}, except for $\delta_n$, 
which are not classical s-numbers and which are only stated here to indicate the flexibility of this notion. The axioms have been verified by Mathé~\cite{Mathe90}, see also~\cite{DHM89,DKLT} or \cite[\S9]{KU26} for more on nonlinear approximation in this context.

\bigskip

Since $h_n$ is the smallest and $a_n$ is the largest s-number, see~\cite[2.3.4 \& 2.6.3]{Pietsch87}, 
the following result implies bounds between arbitrary s-numbers.

\medskip

\begin{theorem}\label{thm:main}
For every bounded linear operator $S \colon X \to Y$ and $n\in\N_0$, we have 
\[
  a_n(S) \;\le\; e\,(n+1)\; h_n(S).
\]
Moreover, if $X$ or $Y$ is a Hilbert space, then
\[
  a_n(S) \;\le\; \sqrt{e\,(n+1)}\;\, h_n(S).
\]
\end{theorem}

\medskip

The proof also gives an explicit form of the rank-$n$ approximation that achieves the bound. 
That is, 
we have $\|S-L\|\le e\,(n+1)\; h_n(S)$ for 
\[
L \,:=\, SA(BSA)^{-1}BS \,\in\, \calL(X,Y)
\]
with contractions $A \in \mathcal{L}(\ell_2^n, X)$ and $B \in \mathcal{L}(Y, \ell_2^n)$ such that the determinant $\det(BSA)$ is (almost) maximized over all contractions. 
It seems interesting that $L$ is of the form $SP$ or $PS$, respectively, where $P$ is a projection, see also~\cite[11.5.2 \& 11.6.2]{Pietsch-ideals}.

\medskip

The dependence on $n$ in Theorem~\ref{thm:main} is optimal. 
For this, it is enough to consider the identity $\id x=x$ as a mapping between $\ell_p$ and $\ell_q$, or its finite-dimensional restrictions, with $p,q\in\{1,2,\infty\}$. 
In fact, 
we only need the known bounds 
\begin{enumerate}[label=(\roman*)]
    \item $\ds a_n(\id\co \ell_1\to\ell_\infty)
    =c_n(\id\co \ell_1\to\ell_\infty)
    = d_n(\id\co \ell_1\to\ell_\infty)
    =\frac{1}{2}$ 
    for $n\ge1$,
    see~\cite[11.11.10 \& 11.5.3 \& 11.6.3]{Pietsch-ideals},  
    \item $\ds h_n(\id\co \ell_1\to\ell_\infty)\asymp \frac{1}{n+1}$, see~\cite{HL84}, 
    \item $\ds a_n(\id\co \ell^m_1\to\ell^m_2)
    =d_n(\id\co \ell^m_1\to\ell^m_2)
    = c_n(\id\co \ell^m_2\to\ell^m_\infty)
    =\sqrt{1-\frac{m}{n}}$\\ 
    for $n\le m$, 
    see~\cite[11.11.8 \& 11.5.2 \& 11.6.2]{Pietsch-ideals}, or \cite{Stechkin54}, and  
    \item $\ds h_n(\id\co \ell^m_1\to\ell^m_2)=h_n(\id\co \ell^m_2\to\ell^m_\infty)\asymp \frac{1}{\sqrt{n+1}}$ for $n< m$, see~\cite{HL84}, 
\end{enumerate}
where we also use $a_n(S)=a_n(S')$ for compact $S$ for (iii), 
and $h_n(S)=h_n(S')$ for $S\in\calL$ for (iv), 
%
see~\cite[11.7.4 \& 11.7.8]{Pietsch-ideals}. 
Moreover, using (i), (iii) and 
\[
b_n(\id\co \ell^m_p\to\ell^m_q)= (n+1)^{\frac1q-\frac1p} \qquad\text{ for $p<q$ and $n< m$, }
\]
see~\cite[Theorem~7.4]{Pietsch-s}, we even obtain that the direct corollary 
\begin{equation}
\max\bset{c_n(S),\, d_n(S)} 
\;\le\; e\, (n+1)\, b_n(S) 
\end{equation}
of Theorem~\ref{thm:main} is sharp up to the constant. 

\medskip

\begin{remark}[Entropy numbers] \label{rem:s-numbers-entropy}
The \emph{entropy numbers} 
of an operator $S\in\calL(X,Y)$ are defined by
\[
 \eps_n(S) \,:=\, 
 \inf_{y_1, \dots, y_{2^{n}} \in Y} \, \sup_{x\in B_X} \, \min_{i=1,\dots,2^n} \|Sx-y_i\|,   
\]
and do only satisfy the conditions 
{\rm (S1)--(S3)}.
The entropy numbers are therefore not an s-number sequence, see~\cite{CP77} or~\cite[6.2.4]{Pietsch07}. 
For example, they satisfy 
$\eps_n(\id_\R)=2^{-n}$ and are therefore never zero, 
although $\rank(I_\R)=1$. 
More generally, we have 
$\eps_n(\id_X)\asymp 2^{-n/m}$ 
for every Banach space $X$ with $\dim(X)=m$, 
see \cite[12.1.13]{Pietsch-ideals}. 
The theory of entropy numbers is also well-established, see e.g.~\cite{CS90} for a comprehensive treatment. 
They are often used to lower bound other s-numbers (or widths) 
by means of \emph{Carl's inequality}, 
$$
\eps_n(S) \,\le\, C_s\, n^{-s} \cdot \sup_{k\le n}\, (k+1)^s\cdot 
\min\bset{c_k(S), d_k(S)}
$$ 
with a constant $C_s>0$ only depending on $s>0$, 
see \cite{Carl81}. 
Moreover, there is the bound 
\[
\max\bset{c_n(S),\, d_n(S)} 
\;\le\; (n+1)\cdot\eps_n(S) 
\]
from \cite[12.3.2]{Pietsch-ideals}. 
Using Theorem~\ref{thm:main} 
and the known $h_n(S)\le 2\,\eps_n(S)$, see \cite[12.3.1]{Pietsch-ideals}, 
we now also obtain 
\begin{equation}
a_n(S) \;\le\; 2e\,(n+1)\; \eps_n(S)
\end{equation}
for all $S\in\calL$, with the corresponding improvements for Hilbert spaces. 
\end{remark}

\goodbreak

\section{A new determinant quantity}\label{sec:delta}

For $k \in \N$ and $S \in \mathcal{L}(X,Y)$, define
\[
  \Delta_k(S) \;:=\; \sup\set{\, |\det(BSA)| \co\; 
  \parbox{42mm}{
  $A \in \mathcal{L}(\ell_2^{k}, X), \norm{A}\le1$,\\[1mm] 
  $B \in \mathcal{L}(Y, \ell_2^{k}), \norm{B} \le 1$}\,
  }, 
\]
and $\Delta_0(S) := 1$, where we identify operators on $\ell_2^k$ with their matrices in the standard basis, as usual.
Since 
$|\det(BSA)| = \prod_{i=0}^{k-1} \sigma_i(BSA) \le \norm{BSA}^k$,
we have $\Delta_k(S) \le \norm{S}^k < \infty$.

\medskip

The $\Delta_k$ are mainly motivated by the 
\emph{Grothendieck numbers} 
\[
  \Gamma_k(S) \,:=\, \sup\Bigl\{\, \bigl|\det\bigl(
  \langle S x_j, b_i \rangle
  \bigr)_{i,j=1}^{k}\bigr|^{1/k}
      \co\, x_1,\dots,x_k \in B_X,\ 
      b_1,\dots,b_k \in B_{Y'} \Bigr\},
\]
which were introduced by Pajor and Tomczak-Jaegermann~\cite[eq.~(2.7)]{PT89} 
in the context of ``volume ratio numbers'' and their relation to Gelfand and entropy numbers, 
see also~\cite{Geiss90} and Remark~\ref{rem:Grothendieck-space}. 
Although $\Delta_n$ does not appear explicitly in~\cite{PT89}, the idea for it clearly stems from reading the proof of their Theorem~3.1.
Note that 
$|\det(T)|$ for 
$T\in\calL(\ell_2^k,\ell_2^k)$ equals the ratio of the Lebesgue measures of $T(B_{\ell_2^k})$ and $B_{\ell_2^k}$. 

The numbers $\Delta_n$ and $\Gamma_n$ are related. 
In fact, it is not hard to see that 
\[
\Delta_{n}(S) \,=\, \Gamma_n(S)^n 
\,=\, \prod_{k=0}^{n-1} h_k(S)
\]
for operators $S\in\calL(H,K)$ for Hilbert spaces~$H,K$, where the~$h_k(S)$ coincide with the singular values~$\sigma_k(S)$.
For general $S\in\calL(X,Y)$, we only have 
\begin{equation}\label{eq:Delta-Gamma}
  \frac{\Gamma_n(S)}{n} \;\le\; \Delta_n(S)^{1/n} \;\le\; \Gamma_n(S). 
\end{equation}
The upper bound is seen by 
$(BSA)_{ij}=\ipr{Sx_j,b_i}$ with 
$x_j=Ae_j \in B_X$ and $b_i=B'e_i \in B_{Y'}$ 
for contractions $A \in \calL(\ell_2^n, X)$, $B \in \calL(Y,\ell_2^n)$, and the \emph{dual} $B'\in\calL(\ell_2^n,Y')$. 
That is, the individual constraints $x_k\in B_X$ and $b_k\in B_{Y'}$ in $\Gamma_n$ 
are replaced in $\Delta_n$ by the joint conditions $\bigl\|\sum_j g_j\, x_j\bigr\| \le \norm{g}_2$ and $\sum_i |\ipr{y,b_i}|^2 \le \norm{y}^2$. 
The other bound follows from 
$A\xi \,:=\, \frac{1}{\sqrt n}\sum_{i=1}^n \xi_i\,x_i$ and 
$By \,:=\, \frac{1}{\sqrt n}\bigl(\ipr{y,b_j}\bigr)_{j=1}^n$ 
being 
contractions for $x_k\in B_X$ and $b_k\in B_{Y'}$.

Both bounds in~\eqref{eq:Delta-Gamma} are sharp, since 
$\Gamma_n(I_{\ell_2})\,=\,\Delta_{n}(I_{\ell_2})^{1/n}   =1$ 
and 
$\Gamma_{n}(\id\co\ell_1^n\to\ell_\infty^n) \,=\, n\, \Delta_n(\id\co\ell_1^n\to\ell_\infty^n)^{1/n} =1$, 
see Proposition~\ref{prop:Delta-Gamma} below. 

\medskip 

Concerning bounds with s-numbers, the $\Gamma_n$ were (implicitly) used, e.g., to prove~\eqref{eq:bound-geom} based on 
\[
\prod_{k=0}^{n-1} c_k(S)
\,\le\, \Gamma_n(S)^n 
\,\le\, n^n \prod_{k=0}^{n-1} h_k(S), 
\]
see~\cite[11.12.3]{Pietsch-ideals} or~\cite{U24}.
In contrast, Lemmas~\ref{lem:ratio-u} and~\ref{lem:ratio-l} below will show, in particular, that the $\Delta_n$ satisfy 
\[
\prod_{k=0}^{n-1} \frac{a_k(S)}{e(k+1)}
\,\le\, \Delta_n(S) 
\,\le\, \prod_{k=0}^{n-1} h_k(S).
\]

\medskip

Actually, we show that consecutive ratios of $\Delta_k(S)$ are bounded from above by 
$h_n$, 
and from below in terms of $a_n$. 
This allows us to get rid of the geometric means in the s-number bound. 

The proof uses the alternative characterization of the Hilbert numbers 
\begin{equation}\label{eq:def-hn}
h_n(S) \;=\; \sup\left\{\, \rho\ge0 \co\; 
  \rho I_{\ell_2^{n+1}} = BSA, \quad
  \parbox{47mm}{
  $A \in \mathcal{L}(\ell_2^{n+1}, X),\ \norm{A}\le1$,\\[1mm]  
  $B \in \mathcal{L}(Y, \ell_2^{n+1}),\ \norm{B} \le 1$}\,\right\}, 
\end{equation}
see~\cite[Satz~5]{Bauhardt}, 
which also implies 
$h_n(S)^{n+1} \le \Delta_{n+1}(S)$ 
for every $S \in \mathcal{L}$ and every $n \in \N_0$, 
since $\det(\rho\, I_{\ell_2^{n+1}}) = \rho^{n+1}$. 
Hence, we obtain from~\cite[6.2.2.2]{Pietsch07} that 
$$\Delta_{n+1}(S)>0 \iff h_n(S)>0 \iff \rank(S)>n.$$

\medskip
\goodbreak

The first lemma shows that consecutive $\Delta_n$ 
decay at least by a factor~$h_n$. 

\begin{lemma}\label{lem:ratio-u}
For every $S \in \mathcal{L}(X,Y)$ and every $n \in \N_0$,
\[
  \Delta_{n+1}(S) \;\le\; h_n(S)\, \Delta_n(S) .
\]
\end{lemma}

\begin{proof}
Let $A \in \mathcal{L}(\ell_2^{n+1}, X)$ and $B \in \mathcal{L}(Y, \ell_2^{n+1})$ be arbitrary contractions, 
and let $\sigma_0 \ge \dots \ge \sigma_{n} \ge 0$ be the singular values of $T := BSA \in \K^{(n+1)\times(n+1)}$. 
Let $T = P \Sigma Q^*$ be a singular value decomposition, 
with $P, Q$ unitary and 
$\Sigma = \diag(\sigma_0, \dots, \sigma_{n})$, 
and let $U, V \in \K^{(n+1) \times n}$ be the first $n$ columns of~$Q$ and~$P$. 
Then, $U^*U = V^*V = I_n$ and $V^* T U = \diag(\sigma_0, \dots, \sigma_{n-1})$. 
Since $AU$ and $V^*B$ are contractions, we get 
\[
  \prod_{i=0}^{n-1} \sigma_i \;=\; |\det(V^* T U)| \;=\; \bigl|\det\bigl((V^*B)\, S\, (AU)\bigr)\bigr| \;\le\; \Delta_n(S) .
\]
Now, since $\Delta_{n+1}(S)=0 \iff h_n(S)=0$, 
we can assume that $\det(T)\neq0$.
Then, $W := \sigma_n\, \Sigma^{-1}$ is a contraction, 
and $(W P^* B)\, S\, (A Q) = W P^* T Q = W \Sigma = \sigma_n\, I_{\ell_2^{n+1}}$, with $W P^* B$ and $A Q$ contractions. 
Hence, we obtain $\sigma_n \le h_n(S)$ from 
the characterization of $h_n$ given in~\eqref{eq:def-hn}. 
This implies
\[
  |\det T| \;=\; \sigma_n \prod_{i=0}^{n-1} \sigma_i \;\le\; h_n(S)\, \Delta_n(S) .
\]
Taking the supremum over all contractions 
$A$ and $B$ yields the result. \\
\end{proof}

\medskip

The next lemma shows that the rate of decay of $\Delta_n$ 
can also be bounded from below in terms of $a_n$.
Note that the proof indeed gives an upper bound on 
$\|S-L\|$ for a specific rank-$n$ operator $L$. 

\begin{lemma} \label{lem:ratio-l}
For all $S\in\calL$ and $n\in\N_0$, we have
\[
\Delta_{n+1}(S)
\;\ge\; \frac{a_n(S)}{e (n+1)}\; \Delta_n(S).
\]
\end{lemma}
 
\begin{proof}
We may assume $\Delta_n(S)>0$.
Let $\eps\in(0,1)$ and choose $A\in\calL(\ell_2^n,X)$ and
$B\in\calL(Y,\ell_2^n)$ with $\|A\|,\|B\|\le1$ and
$\det(BSA)\ge(1-\eps)\,\Delta_n(S)>0$. 
In particular, $BSA$ is invertible.
The operator
\[
L \,:=\, SA(BSA)^{-1}BS \,\in\, \calL(X,Y)
\]
satisfies $\rank(L)\le n$, and therefore $\|S-L\|\ge a_n(S)$.
Choose $x\in B_X$ and $b\in B_{Y'}$ with
$\ipr{(S-L)x,b}\ge(1-\eps)\,\|S-L\|$.
For $\lambda\in(0,1)$, define
$A_0\in\calL(\ell_2^{n+1},X)$ and $B_0\in\calL(Y,\ell_2^{n+1})$ by
\[
A_0(\xi,\tau) \,:=\, \sqrt{1-\lambda}\,A\xi+\sqrt{\lambda}\,\tau x,
\qquad
B_0y \,:=\, \left(\sqrt{1-\lambda}\,By,\; \sqrt{\lambda}\,\ipr{y,b}\right), 
\]
where $\xi\in\ell_2^n$ and $\tau\in\R$. 
By the triangle and Cauchy-Schwarz inequality (applied to~$\R^2$), we obtain 
\[\begin{split}
\|A_0(\xi,\tau)\| 
&\le \sqrt{1-\lambda}\, \|A\xi\|+\sqrt{\lambda}\,|\tau|
\le \bigl(\|\xi\|^2+|\tau|^2\bigr)^{1/2} 
\;=\; \norm{\xi \oplus \tau}_2,
\end{split} 
\]
and
\[
\|B_0y\|^2 = (1-\lambda)\|By\|^2+\lambda\,|\ipr{y,b}|^2 \le \|y\|^2, 
\] 
so that $A_0,B_0$ are contractions.  
Hence, 
$
\Delta_{n+1}(S) \,\ge\, \det(B_0SA_0).
$

The matrix of $B_0SA_0$ has the block form
\[
B_0SA_0 \,=\,
\begin{pmatrix}
(1-\lambda)\, BSA & \sqrt{\lambda(1-\lambda)}\; BSx \\[1mm]
\sqrt{\lambda(1-\lambda)}\; 
\bigl(\ipr{SAe_k,b}\bigr)_{k=1}^n 
& \lambda\,\ipr{Sx,b}
\end{pmatrix}.
\]
Its determinant 
can be computed using the following special case of the Schur determinant formula: 
if $T$ is an invertible $k\times k$ matrix, $u,v\in\ell_2^k$
and $s$ is a scalar, then
\begin{equation}\label{eq:schur}
\det\begin{pmatrix} T & u\\[0.5mm] v^\top & s\end{pmatrix}
\,=\, \det(T)\cdot\bigl(s-v^\top T^{-1}u\bigr).
\end{equation}
This shows that 
\[\begin{split}
\det(B_0SA_0)
\,&=\, (1-\lambda)^n \det(BSA)\cdot
\lambda\,\Bigl(\ipr{Sx,b}-\ipr{SA(BSA)^{-1}BSx,\,b}\Bigr) \\
\,&=\, (1-\lambda)^n\, \lambda\, \det(BSA)\cdot \ipr{(S-L)x,\,b} \\
\,&\ge\, (1-\eps)^2\, (1-\lambda)^n\, \lambda\; \Delta_n(S)\;\norm{S-L}.
\end{split}\]
The function $\lambda\mapsto(1-\lambda)^n\lambda$ is maximized at
$\lambda=\frac{1}{n+1}$ with maximum 
$\frac{n^n}{(n+1)^{n+1}}\ge \frac{1}{e(n+1)}$. 
Hence, $\|S-L\|\ge a_n(S)$ and $\eps\to0$ finish the proof. \\
\end{proof}

\bigskip

We finally show that the lower bound in~\eqref{eq:Delta-Gamma} is sharp. 

\begin{proposition}\label{prop:Delta-Gamma}
For all $n\ge1$, we have
\[
\Gamma_n(\id\co\ell_1^n\to\ell_\infty^n) \,=\, 1
\qquad\text{and}\qquad
\Delta_n(\id\co\ell_1^n\to\ell_\infty^n)^{1/n} \,=\, \frac1n \,.
\]
In particular, the bound $\Gamma_n(S)\le n\,\Delta_n(S)^{1/n}$ is sharp.
\end{proposition}

\begin{proof}
For the lower bound on $\Gamma_n(\id)$, take $x_i:=e_i\in B_{\ell_1^n}$ and
$a_j:=e_j\in B_{\ell_1^n}=B_{(\ell_\infty^n)'}$, so that
$\bigl(\ipr{x_i,a_j}\bigr)$ is the identity matrix.


For the upper bound on $\Delta_n(\id)$, 
we first note that 
for every~$C\in\K^{n\times n}$, 
we have
\begin{equation}\label{eq:sharp-C}
  \|C\|_{\rm HS}^2 \;:=\; \sum_{k=1}^n \norm{Ce_k}_2^2 \;\le\; \norm{C\co \ell_\infty^n\to\ell_2^n}^2, 
\end{equation} 
with the \emph{Hilbert-Schmidt (or 2-summing) norm} $\|C\|_{\rm HS}$, see~\cite[15.5.5]{Pietsch-ideals},
which is an elementary instance of the little Grothendieck theorem. 
Indeed, for independent uniform signs $\varepsilon_k \in \{-1,1\}$, we have 
$\sum_{k} \norm{Ce_k}_2^2 = \mathbb{E}\, \bigl\|\sum_{k} \varepsilon_k\, Ce_k\bigr\|_2^2$, 
as well as $\norm{C\bigl(\sum_{k} \varepsilon_k e_k\bigr)}_2 \le \norm{C\co \ell_\infty^n\to\ell_2^n}$ for all $\varepsilon_k \in \{-1,1\}$, and the inequality follows since the average is smaller than the maximum over $\varepsilon_k$. 

Now, let
$A\co\ell_2^n\to\ell_1^n$ and $B\co\ell_\infty^n\to\ell_2^n$ with
$\|A\|,\|B\|\le1$, 
and apply~\eqref{eq:sharp-C} to $A^\top$ and $B$, 
so that $\|A\|_{\rm HS},\|B\|_{\rm HS}\le1$. 
The AM-GM inequality for the squared singular values of $A$, considered as a map on $\ell_2^n$, yields 
\[
|\det A| \,=\, \prod_{k=1}^n\sigma_k(A)
\,\le\, \Bigl(\frac1n\sum_{k=1}^n\sigma_k(A)^2\Bigr)^{n/2}
\,=\, \Bigl(\frac{\|A\|_{\rm HS}^2}{n}\Bigr)^{n/2}
\,\le\, n^{-n/2}, 
\]
and similarly for $B$. 
Hence, $|\det(BA)|=|\det B|\,|\det A|\le n^{-n/2}$. 
Since this holds for all $A$ and $B$, we obtain $\Delta_n(I_{\ell_1})\le n^{-n/2}$. \\
\end{proof}

\medskip

\begin{remark}[Grothendieck numbers of a space]\label{rem:Grothendieck-space} 
The Grothendieck numbers 
of a space~$X$, 
obtained from the above definition of~$\Gamma_n(S)$ with $S=I_X$, 
have been used already by Grothendieck~\cite{Grothendieck56}.  
They can be used, e.g., to bound the volume ratio of finite-dimensional spaces or to characterize ``weak Hilbert spaces'', 
see \cite[6.1.11]{Pietsch07} for details and references. 
It would be interesting to see whether the quantity $\Delta_n(I_X)$ plays a similar role. 
\end{remark}

\section{Proof of Theorem~\ref{thm:main}}
\label{sec:proof}

First, note that $\rank(S)\le n$ implies $a_n(S)=0$, and there is nothing to prove. 
So assume $\rank(S)>n$ throughout, which also implies $\Delta_n(S)>0$.

The proof of the first (general) part of Theorem~\ref{thm:main} follows immediately from Lemmas~\ref{lem:ratio-u} and~\ref{lem:ratio-l} since 
\[
h_n(S)\,\Delta_n(S)
\,\ge\, \Delta_{n+1}(S)
\,\ge\, 
\frac{a_n(S)}{e (n+1)}\;
\Delta_n(S),
\]
and dividing by $\Delta_n(S)$ finishes the proof. 

\medskip 

If the domain or the codomain is a Hilbert space, 
then some modifications are needed in the proof. 
However, it turns out that, by slightly changing 
$\Delta_n$ to reflect the Hilbert structure, 
these modifications are minimal. 

\medskip

Let first $Y$ be a Hilbert space, 
and recall that $\ipr{\cdot,\cdot}_Y$ denotes its inner product (while $\ipr{x,b}$ denotes the duality pairing of $x\in X$ and $b\in X'$),  
and that 
$B^*\in\calL(\ell_2^n,Y)$ is the adjoint of $B\in \mathcal{L}(Y, \ell_2^{k})$, defined by
$\ipr{By,\xi}_Y=\ipr{y,B^*\xi}_{\ell_2^n}$ for $y\in Y$ and $\xi\in\ell_2^n$. 
We define
\[
  \Delta^H_k(S) \;:=\; 
  \sup\set{\, |\det(BSA)| \co\; 
  \parbox{45mm}{
  $A \in \mathcal{L}(\ell_2^{k}, X),\  \norm{A}\le1$,\\[1mm] 
  $B \in \mathcal{L}(Y, \ell_2^{k}), \ BB^* = I_{\ell_2^k}$}\,
  }. 
\]
That is, $\Delta_k(S)$ with the supremum restricted to pairs whose $B \colon Y \to \ell_2^k$ is a coisometry (i.e., $BB^* = I_k$), 
which means that the functionals $B^*e_1, \dots, B^*e_k$ 
are orthonormal. Coisometries are contractions, so that $\Delta^H_k(S) \le \Delta_k(S)$, and they are preserved under compressions, since $(V^*B)(V^*B)^* = V^*(BB^*)V = I_m$ for $V$ with orthonormal columns. Hence, the proof of Lemma~\ref{lem:ratio-u} applies verbatim to $\Delta^H_k$, 
i.e., we have 
\[
\Delta^H_{n+1}(S) \,\le\, h_n(S)\,\Delta^H_n(S). 
\] 
It remains to show 
\begin{equation}\label{eq:hilbert-growth}
  \Delta^H_{n+1}(S) \;\ge\; 
   \frac{a_n(S)}{\sqrt{e (n+1)}}\; \Delta^H_n(S).
\end{equation}

Again, $a_n(S)>0$ implies $\Delta^H_{n+1}(S)>0$. 
Let $\eps\in(0,1)$ and choose $A,B$ with
$|\det(BSA)|\ge(1-\eps)\,\Delta_n^H(S)>0$, and put
$L:=SA(BSA)^{-1}BS$ with $\rank(L)\le n$, so that $\|S-L\|\ge a_n(S)$. 
The key observation is that
\[
B\,(S-L) \,=\, BS-BSA(BSA)^{-1}BS \,=\, 0,
\]
i.e., the range of $S-L$ is contained in the kernel~$\ker(B)$ of $B$.
Choose $x\in B_X$ with $\|(S-L)x\|\ge(1-\eps)\|S-L\|>0$ and set
$b:=\frac{(S-L)x}{\|(S-L)x\|}
\in\ker(B)\subset Y$. 
This implies that $\ipr{b,B^*e_i}_Y=\ipr{Bb,e_i}_{\ell_2^n}=0$, i.e., 
that $b$ is orthogonal to the functionals represented by (the rows of) $B$. 
For $\lambda\in(0,1)$, define
\[
A_0(\xi,\tau) \,:=\, \sqrt{1-\lambda}\,A\xi+\sqrt{\lambda}\,\tau x,
\qquad
B_0y \,:=\, \bigl(By,\; \ipr{y,b}_Y\bigr).
\]
(Compared to above, the weights 
are needed only on one side.)

Then $\|A_0\|\le1$ as before, and $B_0$ is again a coisometry, since 
$B_0^*e_k=B^*e_k$, $k=1,\dots,n$, and $B_0^*e_{n+1}=b$ are orthonormal. 
We have 
\[
B_0SA_0 \,=\,
\begin{pmatrix}
\sqrt{1-\lambda}\, BSA & \sqrt{\lambda}\; BSx \\[1mm]
\sqrt{1-\lambda}\; 
\bigl(\ipr{SAe_k,b}_Y\bigr)_{k=1}^n 
& \sqrt{\lambda}\,\ipr{Sx,b}_Y
\end{pmatrix}.
\]
Hence, 
using \eqref{eq:schur}, we obtain similar to above that 
\[\begin{split}
\Delta^H_{n+1}(S) 
\,&\ge\, |\det(T_0)|
\,=\, (1-\lambda)^{n/2}\,\lambda^{1/2}\,|\det(BSA)|\cdot
\bigl|\ipr{(S-L)x,\,b}_Y\bigr| \\
\,&\ge\, (1-\eps)^2\,(1-\lambda)^{n/2}\,\lambda^{1/2}\;
\Delta_n^H(S)\; \|S-L\|.
\end{split}\]
Together with~\eqref{eq:hilbert-growth}, this shows
\[
h_n(S) \,\ge\, \frac{\Delta^H_{n+1}(S)}{\Delta^H_{n}(S)}
\,\ge\, (1-\eps)^2\,(1-\lambda)^{n/2}\,\lambda^{1/2}\; \|S-L\|, 
\]
and the result follows from $\lambda=\frac{1}{n+1}$, 
$a_{n}(S)\le\|S-L\|$, 
and $\eps\to0$.
 
Let now $X$ be a Hilbert space. 
In this case, we restrict to isometries $A\co\ell_2^n\to X$ and $\|B\|\le1$. For near-extremal $A$ and $B$, and $L$ as above, we
now have
\[
(S-L)\,A \,=\, SA-SA(BSA)^{-1}BSA \,=\, 0,
\]
i.e., $S-L$ vanishes on $\mathrm{ran}(A)$, so that
$$\sup\bigl\{\|(S-L)x\|\co x\in B_X,\; x\perp\mathrm{ran}(A)\bigr\}
=\|S-L\|\ge a_n(S).$$
Choose such a (normalized) $x$ with $\|(S-L)x\|\ge(1-\eps)\|S-L\|$ and $b\in B_{Y'}$
with $\ipr{(S-L)x,b}_Y=\|(S-L)x\|$, and define
\[
A_0(\xi,\tau) \,:=\, A\xi+\tau x,
\qquad
B_0y \,:=\, \bigl(\sqrt{1-\lambda}\,By,\; \sqrt{\lambda}\,\ipr{y,b}_Y\bigr).
\]
Then, $\|B_0\|\le1$, and $A_0$ is an isometry, since 
$A_0e_k=Ae_k$, $k=1,\dots,n$, and $A_0e_{n+1}=x$ are orthonormal. 
The remaining proof is identical.

\bigskip

\section{Widths of nonsymmetric convex sets} \label{sec:sets}

As indicated above, s-numbers of an identity $I_X$ can be understood as $n$-width of the unit ball $B_X$ in $Y$. It is clearly of interest to extend the aforementioned bounds to widths of nonsymmetric convex sets, see e.g.~\cite{Pinkus,Tikh60}. 
However, it seems that there is still no general/axiomatic theory of $n$-widths that allows for determining the smallest and largest among them. 
Here, we follow the approach of~\cite{KNU} to define ``s-numbers of $S\in\calL(X,Y)$ on $F\subset X$'' 
which aim on approximation $Sf$ only for $f\in F$, 
see also the references given in~\cite{KNU}. 
(One may also call them ``widths of $F$ with respect to a mapping $S$''.)


Let $S\in\calL(X,Y)$ and let $F\subseteq X$ be nonempty and convex with
$S(F)$ bounded. Then, we define by 
\begin{align*}
a_n(S,F) \,&:=\, \inf\set{\sup_{f\in F}\,\|Sf-Lf-y_0\| \co y_0\in Y,\; L\in\calL(X,Y),\; \rank(L)\le n},\\[2pt]
b_n(S,F) \,&:=\, \sup\Bigl\{r>0 \co x+\set{v\in V\co\|Sv\|\le r}\subseteq F
\text{ for some } x\in F \text{ and }\\[-2pt]
&\hspace{27mm} (n{+}1)\text{-dimensional } V\subseteq X
\text{ on which } S \text{ is injective}\Bigr\},\\[2pt]
c_n(S,F) \,&:=\, \inf_{b_1,\dots,b_n\in X'}\, \sup\set{\tfrac12\|Sf-Sg\| \co f,g\in F,\; 
\ipr{f,b_i}=\ipr{g,b_i} \text{ for } i\le n},\\[2pt]
d_n(S,F) \,&:=\, \inf\set{\sup_{f\in F}\,
\, \inf_{y \in V}\, \norm{Sf-y}\co
V\subseteq Y \text{ affine subspace},\; \dim(V)\le n},\\[2pt]
h_n(S,F) \,&:=\, \sup\Bigl\{a_n(BSA) \co
A\in\calL(\ell_2^{n+1},X) \text{ and } x\in F \text{ with } \\[-2pt]
&\hspace{22mm} A(B_{\ell_2^{n+1}})+x\subseteq F, \;
\text{ and } B\in\calL(Y,\ell_2^{n+1}) \text{ with } \|B\|\le1\Bigr\}, 
\end{align*}
the approximation, Bernstein, Gelfand, Kolmogorov and Hilbert numbers of~$S$ on $F$. 
%
The approximation, Bernstein and Kolmogorov numbers of the identity $S=\id_Y$ on~$F\subset Y$ are commonly known as \emph{linear, Bernstein and Kolmogorov widths} of $F$ in $Y$, see~\cite{Pinkus}. 
The latter two are not affected by introducing the more complicated notation with pairs $(S,F)$, 
since the 
$b_n$ and $d_n$ depend on $(S,F)$ only
through the image $S(F)\subseteq Y$. 
For $a_n$, $c_n$ and $h_n$, however, the
functionals $b_i\in X'$ act on the inputs $f\in F$. 
Hence, the operator~$S$ (which we often just consider to be $\id\co X\to Y$) gives an additional degree of freedom to specify the \emph{admissible measurements}, 
see~\cite{KNU} for relations to \emph{minimal errors} of different approximation schemes. 

\medskip

As for the s-numbers, we have
\[
h_n(S,F) \,\le\, b_n(S,F) \,\le\, c_n(S,F) \,\le\, a_n(S,F)
\quad\text{and}\quad
d_n(S,F) \,\le\, a_n(S,F),
\]
see e.g.~\cite[Proposition~3.2]{KNU}.
We now prove a reverse bound similar to Theorem~\ref{thm:main}.
Here, the scalar field is $\mathbb{R}$ 
to avoid technicalities with convexity. 

\begin{theorem}\label{thm:sets} 
Let $S\in\calL(X,Y)$ for real normed spaces $X$ and $Y$. 
For every convex $F\subseteq X$ with $S(F)$ bounded and every
$n\in\N_0$, we have
\[
a_n(S,F) \,\le\, 2\,e^{3/2}(n+1)^{3/2}\, h_n(S,F) 
\]
and, if $Y$ is a Hilbert space, then
\[
a_n(S,F) \,\le\, 2\,e\,(n+1)\, h_n(S,F). 
\]
If $F$ is symmetric, then $a_n(S,F)\le e\,(n+1)\,h_n(S,F)$ in general, 
and if $Y$ is additionally a Hilbert space, then $a_n(S,F)\le\sqrt{e\,(n+1)}\;h_n(S,F)$.
\end{theorem}

The dependence on~$n$ is again sharp, see Remark~\ref{rem:simplex}. 
For Hilbert spaces~$Y$, and with $a_n$ replaced by $c_n$ as well as $h_n$ replaced by $b_n$, a bound without the factor $2e$ was shown in~\cite[Theorem~3.6]{KNU}. 
See also \cite[Theorem~3.3]{KNU} for 
a bound with geometric means, similar to~\eqref{eq:bound-geom}.

\medskip 
\begin{proof}
As in Section~\ref{sec:proof}, we only need minor modifications of Lemmas~\ref{lem:ratio-u}
and~\ref{lem:ratio-l}. 
We call $(A,B)$ \emph{admissible of order $k$ with shift $x$} if
$A\in\calL(\ell_2^k,X)$ and $x\in F$ with $A(B_{\ell_2^k})+x\subseteq F$,
and $\|B\co Y\to\ell_2^k\|\le1$, and define $\Delta_k(S,F)$ as the supremum
of $|\det(BSA)|$ over such pairs. 

 
 

First, Lemma~\ref{lem:ratio-u} holds verbatim for $\Delta_n(S,F)$, since orthogonal compressions remain admissible with the same shift, and
$\sigma_{n}(BSA)=a_{n}(BSA)\le h_n(S,F)$ holds by definition for admissible pairs
of order $n$, we obtain 
\[
\Delta_{n+1}(S,F)\le h_n(S,F)\,\Delta_n(S,F).
\]
 
Instead of Lemma~\ref{lem:ratio-l}, we prove 
\[
\Delta_{n+1}(S,F) \,\ge\,
\frac{a_n(S,F)}{2\,e^{3/2}(n+1)^{3/2}} \;\Delta_n(S,F),
\]
as well as with exponent $1$ instead of $\tfrac32$ and 
without the factor
$\tfrac12$ for symmetric~$F$. 
Afterwards, we treat Hilbert spaces~$Y$.

Again, 
we can assume 
$\Delta_{n}(S,F)>0$. 
For admissible $(A,B)$ of order~$n$ with shift~$x$ and
$|\det(BSA)|\ge(1-\eps)\Delta_n(S,F)>0$, the operator
$L:=SA(BSA)^{-1}BS$ 
has
$\rank(L)\le n$, so that the definition of $a_n(S,F)$ yields $f\in F$
with 
$\|(S-L)(f-x)\|=\|Sf-Lf-(S-L)x\|>(1-\eps)\,a_n(S,F)$ 
for any given $\eps>0$. 

Choose $u:=\frac{f-x}{2}$ and
$b\in B_{Y'}$ with
$\ipr{(S-L)u,b}=\tfrac12\|(S-L)(f-x)\|>\tfrac{1-\eps}{2}\, a_n(S,F)$, 
and define 
\[
A_0(\xi,\tau) \,:=\, (1-\lambda)\,A\xi+\lambda\,\tau u,
\qquad
B_0y \,:=\, \left(\sqrt{1-\lambda}\,By,\; \sqrt{\lambda}\,\ipr{y,b}\right).  
\]
which 
are admissible of order $n+1$ with shift
$x_0:=(1-\lambda) x+\lambda\,\tfrac{f+x}{2}\in~F$, since $|\tau|\le1$ and
\[
A_0(\xi,\tau)+x_0
\,=\, (1-\lambda)\,(A\xi+x)+\lambda\Bigl(\tfrac{1+\tau}{2}f+\tfrac{1-\tau}{2}x\Bigr)
\,\in\, F
\]
by convexity. 
Now, the determinant identity via \eqref{eq:schur} is unchanged, i.e., we get 
\[\begin{split}
\det(B_0SA_0)
\,&=\, (1-\lambda)^{3n/2}\lambda^{3/2}\, \det(BSA)\cdot \ipr{(S-L)u,\,b} \\
\,&\ge\, (1-\eps)^2\, (1-\lambda)^{3n/2}\lambda^{3/2}\; \Delta_n(S,F)\;\frac12\,a_n(S,F).
\end{split}\]
The factor 
$(1-\lambda)^{3n/2}\lambda^{3/2}$ 
is maximized to 
$\bigl(\tfrac{n^n}{(n+1)^{n+1}}\bigr)^{3/2} \ge\frac{1}{(e(n+1))^{3/2}}$ for $\lambda:=\tfrac{1}{n+1}$. 
With $\eps\to0$, we get the result.

Now, let $Y$ be a Hilbert space. As in 
Section~\ref{sec:proof}, 
we restrict to admissible pairs whose
$B$ is a coisometry.
Since $B(S-L)=0$, as before, the vector 
$b:=\tfrac{(S-L)u}{\|(S-L)u\|}$ lies in $\ker(B)$, so that 
$\ipr{(S-L)u,b}_Y=\tfrac12\|(S-L)(f-x)\|>\tfrac{1-\eps}{2}\, a_n(S,F)$,  
and $B_0y:=\bigl(By,\,\ipr{y,b}_Y\bigr)$ is a coisometry.
So, again, the weights in $B_0$ disappear, while $A_0$ remains unchanged. 
The extracted factor becomes $(1-\lambda)^{n}\lambda$, and we conclude as above.

If $F$ is symmetric, this is 
literally as in the case $F=B_X$ treated in 
Lemma~\ref{lem:ratio-l}. 
In fact, all shifts may be taken $0$, as well as $u:=f$ and $b\in B_{Y'}$ with $\ipr{(S-L)f,b}=\|(S-L)f\|>(1-\eps)\,a_n(S,F)$,
i.e., without the factor~$\tfrac12$, and $(A_0,B_0)$ as before. 
This finishes the proof. \\
\end{proof}

\bigskip

The domain-side improvement of Theorem~\ref{thm:main} has no
direct analogue for general $F$, as the weights 
in the construction of $A_0$ are
dictated by the convexity of $F$, 
and not by $\|A_0\|\le1$. 
However, we may consider \emph{ellipsoids} in general normed spaces, defined by images of 
unit balls of Hilbert spaces under linear operators.
We again refer to~\cite[Theorem~3.3]{KNU}

\begin{corollary}\label{cor:ellipsoid} 
Let $S\in\calL(X,Y)$ for real normed spaces $X$ and $Y$, $H$ be a real Hilbert space, $R\in\calL(H,X)$ and $F:=R(B_H)$. 
Then, for all $n\in\N_0$,
\[
a_n(S,F) \,=\, a_n(SR)
\qquad\text{and}\qquad
h_n(S,F) \,=\, h_n(SR),
\]
and hence
$a_n(S,F)
\le\sqrt{e\,(n+1)}\;h_n(S,F)$.
If additionally $Y$ is a Hilbert space, then $a_n(S,F)=h_n(S,F)$.
\end{corollary}

\begin{proof}
We may assume that $R$ is injective, replacing $H$ by $(\ker R)^\perp$
otherwise. Since $F$ is symmetric, the shifts in $h_n(S,F)$ are redundant
and the affine maps in $a_n(S,F)$ may be taken linear. 

For $a_n$, we use that for every $L\in\calL(X,Y)$ we have
$\sup_{f\in F}\|Sf-Lf\|=\|(S-L)R\|$ and
$\rank(L R)\le\rank(L)$, so that $a_n(S,F)\ge a_n(SR)$.
Conversely, let 
$L\in\calL(H,Y)$ with 
$L=\sum_{i=1}^{n}\ipr{\,\cdot\,,h_i}_H\,y_i$  
and 
$\|SR-L\|\le a_n(SR)+\eps$. 
As $R$ is injective, $R'(X')$ is dense in~$H$, and we may choose $b_i\in X'$ with
$\sum_{i=1}^n\|R'b_i-h_i\|\,\|y_i\|\le\eps$. 
Then, 
$L_0:=\sum_{i=1}^n\ipr{\cdot,b_i}\,y_i\in\calL(X,Y)$ has $\rank(L_0)\le n$
and
$\sup_{f\in F}\|Sf-L_0 f\|=\|SR-L\|+\eps\le a_n(SR)+2\eps$,
so that $a_n(S,F)\le a_n(SR)$.

For~$h_n$, note that a pair $(A,B)$ with $A(B_{\ell_2^{n+1}})\subseteq F$ and $\|B\|\le1$ is
precisely of the form $A=R\tilde A$ with
$\tilde A:=R^{-1}A\in\calL(\ell_2^{n+1},H)$ and $\|\tilde A\|\le1$. 
Since $BSA=B(SR)\tilde A$, we obtain $h_n(S,F)=h_n(SR)$.
The bounds now follow from Theorem~\ref{thm:main} applied to
$SR\in\calL(H,Y)$, and the final statement from the coincidence of all
s-numbers for operators between Hilbert spaces.
\end{proof}
 

\begin{remark}[Sharpness and the simplex]\label{rem:simplex}
In an upcoming paper \cite{simplex}, some widths (in $\ell_\infty$) of the regular simplex are determined. In particular, this shows that the exponent $3/2$ in Theorem~\ref{thm:sets} cannot be improved for the Hilbert numbers. 
This does not decide whether the exponent can be lowered when $h_n$ is replaced by the larger Bernstein numbers: in analogy with the Mityagin-Henkin conjecture, Novak~\cite{Novak95} conjectured that a factor of order~$n$ suffices for a comparison of~$c_n$ and~$b_n$ on general convex sets, while Theorem~\ref{thm:sets} only gives $(n+1)^{3/2}$. See also the discussion in~\cite[Section~6]{KNU}.
\end{remark}


\bigskip


\end{document}